\documentclass[10pt]{amsart}

\usepackage[lite]{amsrefs}

\usepackage{amssymb}
\usepackage[T1]{fontenc}
\usepackage{mathrsfs}  
\usepackage[all,cmtip]{xy}

\renewcommand{\P}{\mathbb P}

\renewcommand \phi\varphi

\numberwithin{equation}{section}

\theoremstyle{plain} 
\newtheorem{thm}[equation]{Theorem}

\newtheorem{lem}[equation]{Lemma}
\newtheorem{prop}[equation]{Proposition}

\theoremstyle{definition}
\newtheorem{defn}[equation]{Definition}

\theoremstyle{remark}

\newtheorem{notation}[equation]{Notation}

\begin{document}

\title{A presentation theorem for smooth projective schemes over discrete valuation rings}

\author{Ning Guo}
\author{Ivan Panin}
\address{St. Petersburg branch of V. A. Steklov Mathematical Institute, Fontanka 27, 191023 St. Petersburg, Russia}
\email{guo.ning@eimi.ru}
\email{paniniv@gmail.com}
\date{\today}

\begin{abstract}
In this article, we give a proof for a geometric presentation theorem for any irreducible scheme $X$ smooth projective over a discrete valuation ring $R$. 
As a consequence, for any reductive $R$-group scheme $\mathbf{G}$, we prove that any generically trivial principal $\mathbf{G}$-bundle over $X$ glued to a principal $\mathbf{G}_U$-bundle over the affine line $\mathbb{A}^1_U$ for a semilocal affine scheme $U$.
\end{abstract}
\maketitle






\section{Introduction}
Let $R$ be a discrete valuation ring and  let $X$ be an irreducible scheme smooth over $R$. 
When $R$ contains a field and $X$ is projective over $R$, in \cite{Pan20a} and \cite{PSV15}, two geometric presentation theorems were proved respectively. 
The goal of this article is to remove the equicharacteristic condition on $R$.
The following Theorems \ref{Moving_for_DVR} and \ref{MainHomotopy} are the main results: the first is the geometric presentation theorem and the second is its application to descending principal bundles under reductive $R$-group schemes and in particular the Grothendieck--Serre conjecture concerning principal bundles, see \cite{Pan19a}, \cite{Pan19b}.
\vskip 0.1cm
\begin{defn}
For a regular scheme $S$, let $\mathbf{Sm}_S$ be the category of $S$-smooth schemes. 
A \emph{Nisnevich square} in $\mathbf{Sm}_S$ is a commutative diagram of the  form 
\[
\xymatrix{
    W \ar@{^{(}->}[d] &  V \ar[l] \ar@{^{(}->}[d] \\
    X  & Y \ar[l]
    }
\]
that is Cartesian and such that two vertical arrows are open embeddings, two horizontal arrows are \'etale, and there is an isomorphism $(X- W)_{\mathrm{red}}\cong (Y- V)_{\mathrm{red}}$.
\end{defn}

Let $S$ be an affine scheme and $O_S=\Gamma(S,\mathcal O_S)$. 
In this paper, an \emph{affine} $S$-scheme is a morphism 
$T\to S$ with an affine scheme $T$ such that $O_T$ is a finitely generated $O_S$-algebra. 

\begin{thm}\label{Moving_for_DVR}
Let $R$ be a discrete valuation ring and let $X$ be an $R$-smooth $R$-projective irreducible $R$-scheme of pure
relative dimension $r$. Let ${\bf x}\subset X$ be a finite subset of closed points. Put 
$U=\mathrm{Spec} (\mathcal O_{X,{\bf x}})$.
Let $Z\subset X$ be a closed subset 
avoiding all generic
points of the closed fibre of $X$.  
Then there is 
an affine $U$-smooth scheme $\mathcal X$ of relative dimension one
with a closed subset ${\mathcal Z}\subset {\mathcal X}$ 
fitting into the following commutative diagram
\begin{equation}
\label{diag:A1_homotopy}
    \xymatrix{
     \mathbb A^1_U - \tau(\mathcal Z) \ar[d]_{\mathrm{inc}_{\P}}           &&  \mathcal X- \mathcal Z \ar[ll]^{\tau|_{\mathcal X- \mathcal Z}} \ar[d]_{\mathrm{Inc}} \ar[rr]^{p_{X-Z}} && X-Z \ar[d]_{\mathrm{inc}} &  \\
     \mathbb A^1_U                 && \mathcal X \ar[ll]^{\tau}  \ar[rr]^{p_X} && X  &\\
                          && \mathcal Z \ar[llu]^{\tau|_{\mathcal Z}} \ar[u]^{I} \ar[rr]^{p_Z} && Z \ar[u]^{i} &\\    }
\end{equation}
and subjecting the conditions listed below \\
(i) $\tau$ is an \'{e}tale morphism such that $\tau|_{\mathcal Z} : {\mathcal Z}\to \mathbb  A^1_U$ is a closed embedding;\\
(ii) the left square is a Nisnevich square in $\mathbf{Sm}_U$;\\
(iii) both right-hand side squares are commutative; \\
(iv) the structure morphism $p_U : {\mathcal X}\to U$ defined as $\mathrm{pr}_U\circ \tau$ permits a $U$-section
$\Delta: U \to \mathcal X$;\\
(v) we have $p_X \circ \Delta=\mathrm{can}$, where $\mathrm{can}: U \hookrightarrow X$ is the open immersion;\\
(vi) the $i_0:=\tau \circ \Delta$ is the zero section $i_0: U \to \mathbb A^1_U$
of the projection $\mathrm{pr}_U : \mathbb A^1_U\to U$; \\
(vii) $\mathcal Z$ is $U$-finite; \\
(viii) $\tau(\mathcal Z)\cap \{1\}\times U=\emptyset$.
\end{thm}
The following descent of principal bundles is a direct consequence of Theorem \ref{Moving_for_DVR}. 
By using this result, we prove the Grothendieck--Serre on smooth projective schemse over DVRs, see \cite{GPan3} and \cite{PanSt}*{Thms.~1.1 and 1.2}.

\begin{thm}
\label{MainHomotopy}
Let $R$, $X$, ${\bf x}\subset X$, $r$, $U=\mathrm{Spec} (\mathcal O_{X,\bf x})$ be as in Theorem \ref{Moving_for_DVR}. 
Let ${\bf G}$ be a reductive $R$-group scheme and 
$\mathcal G/X$ be a principal ${\bf G}$-bundle 
trivial over the generic point $\eta$ of $X$.
Then there exist a principal ${\bf G}$-bundle $\mathcal G_t$ over $\mathbb A^1_U$
and a closed subset $\mathcal Z'$ in $\mathbb A^1_U$ finite over $U$
such that 
\begin{itemize}
\item[(i)] the $\mathbf{G}$-bundle $\mathcal{G}_t$ is trivial over the open subscheme $\mathbb A^1_{\mathcal Z'}-\mathcal Z'$ in $\mathbb A^1_U$;
\item[(ii)] the restriction of $\mathcal{G}_t$ to $\{0\}\times U$ coincides with the restriction of $\mathcal{G}$ to $U$ via $\mathrm{can}^{\ast}$; and
\item[(iii)]  $\mathcal{Z}'\cap (\{1\}\times U)=\emptyset$.
\end{itemize}
\end{thm}
\begin{proof}[Proof of Theorem \ref{MainHomotopy}]
The principal ${\bf G}$-bundle $\mathcal G$ is 
trivial over the generic point $\eta$ of $X$. 
The main theorem of \cite{Guo22} yields the triviality of the principal ${\bf G}$-bundle $\mathcal G|_U$.
In turn, this shows that there is a closed subset
$Z\subset X$ avoiding all generic
points of the closed fibre of $X$ 
and such that
the principal ${\bf G}$-bundle $\mathcal G|_{X-Z}$ is trivial.
Now consider the diagram \eqref{diag:A1_homotopy} as in
Theorem \ref{Moving_for_DVR} 
and take the closed subscheme 
${\mathcal Z}'=\tau(\mathcal Z)$
of the scheme $\mathbb A^1_U$. 

Clearly, $\mathcal Z'$ is finite over $U$ and it enjoys the property (iii).
To construct a desired principal ${\bf G}$-bundle $\mathcal G_t$ over $\mathbb A^1_U$
note that the bundle $p^*_X({\mathcal E})|_{\mathcal X- \mathcal Z}$ is trivial. Since the left square is a Nisnevich square in $\mathbf{Sm}_U$, there is a principal ${\bf G}$-bundle $\mathcal G_t$ over $\mathbb A^1_U$
such that 
\begin{itemize}
\item[1)]$\tau^*(\mathcal G_t)=p^*_X({\mathcal E})$;
\item[2)] $\mathcal G_t|_{\mathbb A^1_U-\mathcal Z'}$.
\end{itemize}    
Now the properties (v) and (vi) of the diagram \eqref{diag:A1_homotopy}
show that 
the principal ${\bf G}$-bundle $\mathcal G_t$ 
enjoys the property (ii), that is, the restriction of $\mathcal G_t$ to $\{0\}\times U$ coincides
with the restriction of $\mathcal G$ to $U$.

The proof of the theorem is completed.
\end{proof}

\subsection*{Acknowledgements} 
The authors thank the excellent environment of the International Mathematical Center at POMI.
This work is supported by Ministry of Science and Higher Education of the Russian Federation, agreement \textnumero~ 075-15-2022-289. 

\section{Weak elementary fibration: Definition and Result}\label{s:w_el_fibr_inf field}
In this Section, we recall the notion of a weak elementary fibration introduced in \cite[Defn.~1.1]{GPan2}
and formulate the main results of \cite[Thms.~1.3 and 1.4]{GPan2} as Theorems 
\ref{Very_General} and \ref{Very_General_Diagram} below.

The notion of a weak elementary fibration is a modification of the one introduced by Artin in~\cite[Exp.~XI, D\'ef.~3.1]{Art73}
and of its version introduced in \cite[Defn.~2.1]{PSV15}.
The result \cite[]{GPan2} is a modification of a result due to M. Artin~\cite{Art73} concerning existence of nice neighborhoods.
 
In this paper, the term {\bf scheme} means a separated, quasi-compact, Noetherian scheme of finite Krull dimension 
(as in \cite{MV99}).

\begin{defn}
\label{DefnElemFib}
A morphism $q: T\to S$ of schemes is called
a \emph{weak elementary fibration} if this morphism can be included in a commutative diagram
\begin{equation}
\label{SquareDiagram}
    \xymatrix{
     T\ar[drr]_{q}\ar[rr]^{j}&&
\bar T\ar[d]_{\bar q}  &&T_{\infty}\ar[ll]_{i}\ar[lld]^{q_{\infty}} &\\
     && S  &\\    }
\end{equation}
satisfying the following conditions:
\begin{itemize}
\item[{\rm(i)}] $j$ is an open embedding dense at each fibre of $\bar q$; $j: T\to \bar T-T_{\infty}$ is an isomorphism;
\item[{\rm(ii)}]
$\bar q$ is smooth projective all of whose fibres are equi-dimensional of dimension one;
\item[{\rm(iii)}] $q_{\infty}$ is a finite \'{e}tale surjective morphism;
\item[{\rm(iv)}] $q=\mathrm{pr}_S \circ \pi$, where $\pi: T\to \mathbb A^1_S$ is a finite flat surjective morphism.
\end{itemize}
The diagram \eqref{SquareDiagram} is called \emph{a diagram of the weak elementary fibration} $q: T\to S$.

Let $Z\subset T$ be a closed subscheme.
A scheme morphism $q: T\to S$ is called
a \emph{weak elementary $Z$-fibration}
if it is a weak elementary fibration and the morphism
$q|_Z: Z\to S$ is finite; particularly, $Z\cap T_{\infty}=\emptyset$.
In this case the diagram \eqref{SquareDiagram}
is called \emph{a diagram of the weak elementary $Z$-fibration} $q: T\to S$.
\end{defn}

\begin{notation}
Let $B$ be a local regular scheme and $b\in B$ its closed point.
For a $B$-scheme $Y$ put $Y_b=Y\times_B b$. For a $B$-morphism
$\phi: Y'\to Y''$ of $B$-schemes write $\phi_b$ for the $b$-morphism
$\phi\times_B b: Y'\times_B b\to Y''\times_B b$.
\end{notation}

\begin{thm}(\cite[Thm.~1.3]{GPan2})\label{Very_General}
Let $B$ be a local regular scheme and $b\in B$ its closed point.
Let $X/B$ be an irreducible $B$-smooth, $B$-projetive scheme of relative dimension $r$.
Let $\mathbf{x} \subset X$ be a finite subset of closed points. 
Then there are an open subscheme $S$ in $\mathbb A^{r-1}_B$,
an open subscheme $X_S\subset X$ containing $\mathbf{x}$, and a weak elementary fibration
$q_S: X_S\to S$.

Let $Z\subset X$ be a closed subscheme such that $\dim Z_b <r$. Then there is an open subscheme $S$ in $\mathbb A^{r-1}_B$,
an open subscheme $X_S\subset X$ containing $\mathbf{x}$ and a weak elementary $Z_S$-fibration
$q_S: X_S\to S$, where $Z_S:=Z\cap X_S$.
\end{thm}

\begin{thm}(\cite[Thm.~1.4]{GPan2})\label{Very_General_Diagram}
Under the hypotheses of Theorem \ref{Very_General}, there is an open
$S\subset \mathbb A^{r}_B$, an open neighborhood $\dot X_S$ of the set ${\bf x}$ in $X$ and
a weak elementary fibration $\dot q_S: \dot X_S\to S$ and a commutative diagram of $S$-schemes
\begin{equation}
\label{V_G_SquareDiagram2}
    \xymatrix{
     \dot X_S\ar[drr]_{\dot q_S}\ar[rr]^{j_S}&&
\hat X_S\ar[d]_{\hat q_S}&& \mathbb W_S\ar[ll]_{i}\ar[lld]^{\mathrm{pr}_S} &\\
     && S  &\\    }
\end{equation}
which is a diagram of the weak elementary fibration $\dot q_S: \dot X_S\to S$.

Moreover, if $Z\subset X$ is as in Theorem \ref{Very_General}, then
the open $S\subset \mathbb A^{r}_B$, the open neighborhood $\dot X_S$ of the set ${\bf x}$ in $X$,
the weak elementary fibration $\dot q_S: \dot X_S\to S$ and
the commutative diagram \eqref{V_G_SquareDiagram2} of $S$-schemes
can be chosen such that the \eqref{V_G_SquareDiagram2}
is a diagram of the weak elementary $\dot Z_S$-fibration
$\dot q_S: \dot X_S\to S$.
\end{thm}

\section{Overcoming finite field obstructions}
\textbf{Panin's Reembedding closed subschemes into affine lines.}
For a field $k$ and a smooth connected $k$-curve $C$, every closed point $x\in C$ can be reembeded into $\mathbb{A}^1_k$ in the sense that the there is a closed point $y\in \mathbb{A}^1_k$ whose residue field $k_y$ is isomorphic to $k_x$. 
Indeed,  $x$ has an open neighborhood $U$ \'etale over $\mathbb{A}^1_k$, such that $k_x/k$ splits into a separable algebraic extension $k_x/k'$ and a simple extension $k'/k$, then primitive element theorem implies that $k_x/k$ is simple.
However, when $k$ is a finite field, say $k=\mathbb{F}_2$, although $\mathbb{A}^1_{\mathbb{F}_2}$ has infinitely many closed points, its closed points of a fixed degree is finite.
For instance, there is no closed embedding from the disjoint union of three copies of $\mathrm{Spec}\, \mathbb{F}_2$ into $\mathbb{A}^1_{\mathbb{F}_2}$, which only has two closed points of degree $1$.
We formulate this phenomena into the following `finite field obstructions'.
This phenomena was discovered by Panin.
He developed a method to overcome finite field obstructions already in \cite{Pan15}, see also \cite{Pan20a}*{\S5--\S6}. 
Much later, some higer dimensional generalizations were done in \cite{Ces22b}*{\S3}.
\vskip 0.1cm
For a finite field $k$ and a finite $k$-scheme $Z_k$, we say that there is no \emph{finite field obstruction} for embedding $Z$ into $\mathbb{A}_k^1$, if for every positive integer $d$, the following inequality holds
\[
\#\{z\in Z_k \,|\, [k(z)\colon k]=d\}\leq \#\{x\in  \mathbb{A}^1_k \,|\, [k(x)\colon k]=d\}.
\]
\begin{lem}\label{extend-points-local}
For a semilocal affine scheme $U$, an irreducible semilocal affine $U$-scheme $S$ whose each closed $U$-fiber is semilocal, a $U$-section $\delta\colon U\to S$ such that $T:= \delta(U)\subset S$ is closed, there exist a finite \'etale morphism $\rho\colon S'\to S$ from a $U$-scheme $S'$ and a lift $\delta'\colon T\to S'$ of $\delta$ such that
\begin{itemize}
\item[(i)] for every closed point $u\in U$, the point $\delta'(\delta^{-1}(u))\in S'_u$ is the unique $k(u)$-rational point of $S'_u$;
\item[(ii)] for every closed point $u\in U$, there is no finite field obstruction for embedding $S'_u$ into $\mathbb{A}_u^1$, that is
\[
\text{$\#\{z\in S'_u\,|\, [k(z)\colon k(u)]=d\}\leq \#\{x\in  \mathbb{A}^1_u \,|\, [k(z)\colon k(u)]=d\}$ for each integer $d\geq 1$}.
\]
\end{itemize}
When $S$ is regular, we may construct $S'$ such that it is irreducible.
\end{lem}
\begin{proof}
Let $\mathbf{t}$ (resp., $\mathbf{t}'$) be the finite set of closed points of $T$ whose residue fields are infinite (resp., finite) and let $\mathbf{s}$ (resp., $\mathbf{s}'$) be the finite set of closed points $S\backslash T$ whose residue fields are infinite (resp., finite). 
For any integer $q>1$, each $s_i\in \mathbf{s}$ and each $s_i' \in \mathbf{s}'$ we choose monic polynomials 
\begin{itemize}
\item $f_{s_i}\in k(s_i)[x]$ as the product of $q$ distinct linear polynomials;
\item $f_{s_i'}\in k(s_i')[x]$ irreducible of degree $q$.
\end{itemize}
Besides, we choose a monic polynomial $f_{\mathbf{t}}\in xR[x]$ of degree $q$ such that 
\begin{itemize}
\item the reduction of $f_{\mathbf{t}}$ in every $k(t_i)[x]$ has degree $q$ as a product of distinct linear polynomials;
\item the reduction of $f_{\mathbf{t}}$ in every $k(t_i')[x]$ is a product of $x$ and a different irreducible polynomial of degree $q-1$.
\end{itemize}
By Nakayama's lemma, there exists a polynomial $F(x)\in \Gamma(S,\mathcal{O}_S)[x]$ of degree $q$ such that the reductions of $F$ become $f_{\mathbf{t}}$ on the section $T$,  become $f_{s_i}$ on $s_i$ and become $f_{s_i'}$ on $s_i'$. 
The estimation above for the number of points of degree $q$ implies that when $q$ is large enough the displayed inequality is satisfied. 
Moreover, the polynomial $F(x)$ determines a scheme $S':= \mathrm{Spec} (\Gamma(S,\mathcal{O}_S)[x]/F(x))$ that is finite \'etale over $S$.
Let $\rho\colon S'\to S$ be the resulting morphism, hence the preimage  $\rho^{-1}(T)=T\sqcup T'$ by construction of $F(x)$.
It suffices to define $\delta'$ as the composite $T\hookrightarrow \rho^{-1}(T)\hookrightarrow S'$, which is a lift of $\delta$.
If $S$ is regular, then so is $S'$.
We may hence replace $S'$ by an irreducible component of $S'$ containing $\delta'(T)$. 
\end{proof}

\begin{lem}\label{overcome}
For a semilocal affine scheme $U$, an affine morphism $p_U\colon \mathcal{X}\to U$ of schemes with a section $\Delta\colon U\to {\mathcal{X}}$, and a $U$-finite closed subscheme $\mathcal{Z}\subset {\mathcal{X}}$, there exists a finite $U$-morphism $\phi\colon \tilde{\mathcal{X}}\to \mathcal{X}$ with a $U$-finite closed subscheme $\tilde{\mathcal{Z}}:= \phi^{-1}(\mathcal{Z}\cup \Delta(U))$ such that 
\begin{itemize}
\item[(i)] $\Delta$ lifts to a section $\Delta'\colon U\to \tilde{\mathcal{X}}$;
\item[(ii)] $\phi$ is \'etale around $\tilde{\mathcal{Z}}$;
\item[(iii)] for every closed point $u\in U$ and every integer $n\geq 1$, we have
\[
\#\{z\in \tilde{\mathcal{Z}}_{k(u)}\,|\, [k(z)\colon k(u)]=n\} \leq  \#\{z\in \mathbb{A}^1_{k(u)}\,|\, [k(z)\colon k(u)]=n\}.
\]
\end{itemize}
\end{lem}
\begin{proof}
Take the set $\mathbf{z}$ of all the closed points of $\mathcal{Z}\cup \Delta(U)$ in the closed $U$-fibers and let $S$ be the semilocalization of $\mathcal{X}$ at $\mathbf{z}$. 
The set $\mathbf{z}$ is finite because $U$ is semilocal.
Since $\mathcal{X}\to U$ is quasi-projective, when forming $S$, by \cite{GLL13}*{Thm.~5.1}, there exists a hypersurface $H\subset \mathbb{P}^N_U$ for an $N\geq 0$ such that $\mathcal{X}\cap H$ contains $\mathcal{Z}\cup \Delta(U)$ but excludes the closed points of $\mathcal{X}-(\mathcal{Z}\cup \Delta(U))$. Therefore, after semilocalization, one may view $\mathcal{Z}\cup \Delta(U)$ as a closed subscheme of $S$ and denote $T:= \Delta(U)$ as the section in $S$.  
Without generalities, we assume that $S$ is irreducible, so Lemma \ref{extend-points-local} applies, yielding a finite \'etale morphism $\tilde{S}\to S$. 
By limit argument and Zariski main theorem, $\tilde{S}$ can by replaced by a $U$-scheme $\tilde{\mathcal{X}}$ satisfying the desired conditions. 
\end{proof}

\section{Arranging a relative curve to an affine line}
The section above avoids the finite finite obstruction for embedding a closed subscheme $\mathcal{Z}\subset \mathcal{X}$ into $\mathbb{A}^1_U$, so one can extend this embedding into a morphism $\mathcal{X}\to \mathbb{A}^1_U$.
\begin{lem}\label{toaffineline}
For a semilocal domain $\mathfrak{R}$, an affine $\mathfrak{R}$-smooth scheme $C$ of pure relative dimension one with an $R$-finite closed subscheme $Z\subset C$.
If for every maximal ideal $\mathfrak{m}\subset \mathfrak{R}$, there is no finite field obstruction for embedding $Z_{\mathfrak{m}}$ to $\mathbb{A}^1_{\kappa_{\mathfrak{m}}}$,  
then there exist a quasi-finite flat surjective morphism $\tau\colon C\to \mathbb{A}_{\mathfrak{R}}^1$ such that $\tau|_{Z}\colon Z\hookrightarrow \mathbb{A}^1_{\mathfrak{R}}$ is a closed embedding and $\tau$ is \'etale on an open neighborhood of $Z$.
\end{lem}
\begin{proof}
Pick a maximal ideal $\mathfrak{m}\subset \mathfrak{R}$, on the smooth $\kappa_{\mathfrak{m}}$-curve $C_{\kappa_{\mathfrak{m}}}$ there are closed points $z_i\in Z_{\kappa_{\mathfrak{m}}}$. The smoothness at $z_i$ yields a factorization $C_{\kappa_{\mathfrak{m}}}\to \mathbb{A}^1_{\kappa_{\mathfrak{m}}}\to \mathrm{Spec}\, \kappa_{\mathfrak{m}}$ such that $\kappa(z_i)$ is a finite separable extension of a simple algebraic extension of $\kappa_{\mathfrak{m}}$, hence a primitive element theorem implies that $\kappa(z_i)/\kappa_{\mathfrak{m}}$ is simple  so that there is a point $\hat{z}_i\in \mathbb{A}^1_{\kappa_{\mathfrak{m}}}$ such that $\kappa(z_i)\simeq \kappa(\hat{z}_i)$. 
Further, since there is no finite field obstruction, we may assume that these $\hat{z}_i$ are distinct.
Now, we write $Z_{\kappa_{\mathfrak{m}}}$ as 
\[
\textstyle \text{$Z_{\kappa_{\mathfrak{m}}}=\coprod_{i=1}^l \mathrm{Spec} (\mathcal{O}_{C_{\kappa_{\mathfrak{m}}}, z_i}/\mathcal{I}_i^{r_i})$\quad where $\mathcal{I}_i$ is the ideal sheaf of $z_i^{\mathrm{red}}$ and $r_i>0$.}
\]
Since $C_{\kappa_{\mathfrak{m}}}$ is $\kappa_{\mathfrak{m}}$-smooth, the completions $\hat{\mathcal{O}}_{C_{\kappa_{\mathfrak{m}}}, z_i}$ are regular complete hence by \cite{Bou98}*{IX, \S3, \textnumero 3, Thm.~1}, the relations $\kappa(\hat{z}_i)\simeq \kappa(z_i)$ yield isomorphisms
\[
\text{$\hat{\mathcal{O}}_{C_{\kappa_{\mathfrak{m}}},z_i}\simeq \hat{\mathcal{O}}_{\mathbb{A}^1_{\kappa_{\mathfrak{m}}}, \hat{z}_i}$\quad and in particular, isomorphisms $\mathcal{O}_{C_{\kappa_{\mathfrak{m}}},z_i}/\mathcal{I}_i^{r_i}\simeq \mathcal{O}_{\mathbb{A}^1_{\kappa_{\mathfrak{m}}}, y_i}/\mathfrak{m}_{\hat{z}_i}^{r_i}$.}
\]
Using these isomorphisms and Nakayama's lemma, we obtain a closed embedding $i\colon Z_{\kappa_{\mathfrak{m}}}\hookrightarrow \mathbb{A}^1_{\mathfrak{R}}$.
Choose a section $s\in \Gamma(C,\mathcal{O}_C)$ that vanishes at a closed point of each irreducible fiber of $C_{\kappa_{\mathfrak{m}}}$ and coincides with the pullback of the coordinate via $i$, then $s$ determines a morphism $\tau\colon C\to \mathbb{A}^1_{\mathfrak{R}}$ such that $\tau|_Z$ is a closed embedding. Further, $\tau$ is flat by the fiberal criterion of flatness \cite{EGAIV2}*{6.1.5} and \cite{EGAIV3}*{11.3.11}.
So, $\tau$ is \'etale along $Z$ since the residue field extensions are trivial. Finally, $\tau$ is quasi-finite because so is $i$ and the quasi-finite locus \cite{SP}*{01TI} of $\tau$ is open.
\end{proof}

\begin{prop}\label{map-to-affine-line} 
For a semilocal domain $\mathfrak{R}$ and its spectrum $U$, an affine $U$-smooth scheme of relative dimension one
\[
\textstyle p_U\colon \mathcal{X}\to U   
\]
 with a $U$-finite closed subscheme defined by a section $f\in \Gamma(\mathcal{X},\mathcal{O}_{\mathcal{X}})$, if $p_U$ fits into the diagram of a weak elementary fibration 
\[
\xymatrix{
    \mathcal{X} \ar[rd]_{p_U} \ar@{^{(}->}[r] & \hat{\mathcal{X}} \ar[d]^{\hat{p}_U} & \mathcal{D} \ar[ld] \ar@{_{(}->}[l] \\
                                               & U                                        &     
                                               }
\]
with a section $\Delta$ of $p_U$ such that for every maximal ideal $\mathfrak{m}\subset \mathfrak{R}$ and every integer $d\geq 1$, we have
\[
\#\{z\in \mathcal{Z}_{\kappa_{\mathfrak{m}}}\,|\, [\kappa(z): \kappa_{\mathfrak{m}}]=d \} \leq \#\{y\in \mathbb{A}^1_{\kappa_{\mathfrak{m}}}\,|\, [\kappa(y): \kappa_{\mathfrak{m}}]=d \}
\]
and the point $\Delta(\{\mathfrak{m}\})$ is the unique $\kappa_{\mathfrak{m}}$-point of $\mathcal{Z}_{\mathfrak{m}}$, then there is a finite surjective $U$-morphism
\[
\tau \colon \mathcal{X}\to \mathbb{A}^1_U
\]
that restricts to a closed embedding $\tau|_{\mathcal{Z}}\colon \mathcal{Z}\hookrightarrow \mathcal{X}$ enjoying the following properties
\begin{itemize}
\item[(i)] $\tau$ is \'etale along $\mathcal{Z}\cup \Delta(U)$;
\item[(ii)] $\tau^{-1}(\tau(\mathcal{Z}))=\mathcal{Z}\sqcup \mathcal{Z}'$ and $\mathcal{Z}'\cap \Delta(U)=\emptyset$;
\item[(iii)] $\tau^{-1}(\{0\}\times U)=\Delta(U)\sqcup \mathcal{D}$ and $\mathcal{D}\cap \mathcal{Z}=\emptyset$;
\item[(iv)] $\tau^{-1}(\{1\}\times U)\cap \mathcal{Z}=\emptyset$;
\item[(v)] there is monic polynomial $h\in R[t]$ such that 
\[
(h)=\ker (\mathfrak{R}[t]\to \Gamma(\mathcal{X},\mathcal{O}_{\mathcal{X}})/(f)).
\]
\end{itemize}
\end{prop}
\begin{proof}
The construction in Lemma \ref{toaffineline} shows that there is a closed embedding $\rho\colon \mathcal{Z}\hookrightarrow \mathbb{A}^1_{U}$.
By assumption, $\mathcal{D}:= \hat{\mathcal{X}}\backslash \mathcal{X}$ is an exceptional Cartier divisor. 
The construction of the elementary fibration implies that $\mathcal{L}:= \mathcal{O}_{\hat{\mathcal{X}}}(\mathcal{D})$ is the pullback of $\mathcal{O}(1)$ on a projective space, so $\mathcal{L}$ is ample.
To find $\tau$, we first construct a finite surjective $U$-morphism $\hat{\tau}\colon \hat{\mathcal{X}}\to \mathbb{P}^1_U$ via suitable sections $t_1, t_{\infty}\in \Gamma(\hat{\mathcal{X}},\mathcal{L}^{\otimes n})$ for a large integer $n>0$.
Let $t_{\infty}$ be a global section of $\mathcal{L}$ whose vanishing locus is exactly $\mathcal{D}$.
To find $t_1$, for the closed fibers $\hat{\mathcal{X}}_{\mathbf{u}}$, we construct $t_{1,\mathbf{u}}\in H^0(\hat{\mathcal{X}}_{\mathbf{u}}, \mathcal{L}^{\otimes n} )$ for an integer $n>0$ satisfying the following conditions
\begin{itemize}
\item[(a)] $t_{1,\mathbf{u}}|_{\mathcal{D}_{\mathbf{u}}}$ has no zeros and $g:= t_{1,\mathbf{u}}/t_{\infty, \mathbf{u}}^n$ is a regular function in $\mathcal{O}(\mathcal{X}_{\mathbf{u}})$  for an integer $n>0$;
\item[(b)] $t_{1,\mathbf{u}}|_{\mathcal{Z}_{\mathbf{u}}}=\rho \cdot t_{\infty, \mathbf{u}}^n|_{\mathcal{Z}_{\mathbf{u}}}$.
\end{itemize}
Let $J\subset \mathcal{O}_{\hat{\mathcal{X}}_{\mathbf{u}}}$ be the ideal sheaf defining $\mathcal{Z}_{\mathbf{u}}\sqcup \mathcal{D}_{\mathbf{u}}$.
We consider the following short exact sequence
\[
0\to J\to \mathcal{O}_{\hat{\mathcal{X}}_{\mathbf{u}}}\to \mathcal{O}_{\mathcal{Z}_{\mathbf{u}}\sqcup \mathcal{D}_{\mathbf{u}}}\to 0.
\]
By Serre's theorem \cite{EGAIII1}*{2.6.1}, there is an integer $n>0$ such that $H^1(\hat{\mathcal{X}}_{\mathbf{u}}, \mathcal{L}^{\otimes n} \otimes J )=0$, so
\[
\text{$H^0(\hat{\mathcal{X}}_{\mathbf{u}}, \mathcal{L}^{\otimes n}|_{\hat{\mathcal{X}}_{\mathbf{u}}}) \twoheadrightarrow H^0(\mathcal{Z}_{\mathbf{u}}, \mathcal{L}^{\otimes n} |_{\mathcal{Z}_{\mathbf{u}}})\oplus H^0(\mathcal{D}_{\mathbf{u}}, \mathcal{L}^{\otimes n}|_{\mathcal{D}_{\mathbf{u}}})$\quad  is surjective.}
\]
Note that $H^0(\mathcal{D}_{\mathbf{u}}, \mathcal{L}^{\otimes n}|_{\mathcal{D}_{\mathbf{u}}})$ is a semilocal ring, it contains an element $\alpha_{\mathbf{u}}$ that has no zeros on $\mathcal{D}_{\mathbf{u}}$. 
By this surjective map, we can lift the relation $t_{1,\mathbf{u}}|_{\mathcal{Z}_{\mathbf{u}}}=\rho\cdot t_{\infty, \mathbf{u}}^n|_{\mathcal{Z}_{\mathbf{u}}}$ over $\hat{\mathcal{X}}_{\mathbf{u}}$.
Therefore, the element $(\rho\cdot t^n_{\infty,\mathbf{u}}|_{\hat{\mathcal{X}}_{\mathbf{u}}}, \alpha_{\mathbf{u}})$ lifts to the desired $t_{1,\mathbf{u}}$ such that $g:= t_{1,\mathbf{u}}/t^n_{\infty, \mathbf{u}}$ is a regular function in $\mathcal{O}(\mathcal{X}_{\mathbf{u}})$, the restriction $t_{1,\mathbf{u}}|_{\mathcal{D}_{\mathbf{u}}}$ has no zeros, and the condition (2) is satisfied by construction.

It remains to extend $t_{1,\mathbf{u}}$ to a section $t_{1}\in \Gamma(\hat{\mathcal{X}}, \mathcal{L}^{\otimes n})$.
By Serre's criterion and enlarging $n$ if necessary, the group $H^1(\hat{\mathcal{X}}, I_{\mathbf{u}}(n))=0$ for the ideal sheaf $I_{\mathbf{u}}$ defining $\hat{\mathcal{X}}_{\mathbf{u}}\subset \hat{\mathcal{X}}$ so the map
\[
\text{$H^0(\hat{\mathcal{X}}, \mathcal{L}^{\otimes n} )\twoheadrightarrow H^0(\hat{\mathcal{X}}_{\mathbf{u}}, \mathcal{L}^{\otimes n} )$\quad is surjective.}
\]
It suffices to take $t_1$ as the preimage of $t_{1,\mathbf{u}}$ under this surjection.
Since $\mathcal{D}$ is finite over $U$, it is semilocal. Hence $t_1$ has no zeros on $\mathcal{D}$ since $t_{1,\mathbf{u}}$ has no zeros on $\mathcal{D}_{\mathbf{u}}$.
Therefore, the morphism
\[
\text{$\hat{\tau}\colon \hat{\mathcal{X}}\overset{[t_{\infty}^n\colon t_1]}{\longrightarrow} \mathbb{P}^1_U$\quad is well-defined.}
\]
Finally, the restriction $\hat{\tau}|_{\mathcal{Z}}$ is the closed embedding $\mathcal{Z}\hookrightarrow \mathbb{A}^1_U$ by Nakayama' s lemma.
By construction, the restriction $\tau:= \hat{\tau}|_{\mathcal{X}}\colon \mathcal{X}\to \mathbb{A}^1_U$ is finite and surjective.
\end{proof}

\begin{proof}[Proof of Theorem \ref{Moving_for_DVR}] 
We start with an irreducible $R$-smooth projective scheme $X$ of pure relative dimension $r$ over a discrete valuation ring $R$, a finite subset $\mathbf{x}$ of closed  points of $X$, a closed subset $Z\subset X$ that avoids all generic points of the closed fibre of $X$.
Then by Theorem~\ref{Very_General_Diagram}, there are an open $S\subset \mathbb A^r_R$, an open neighborhood $\dot{X}_S$ of $\mathbf{x}$ in $X$, a weak elementary fibration $\dot{q}_S\colon \dot{X}_S\rightarrow S$ and a commutative diagram 
\begin{equation}
    \xymatrix{
     \dot X_S\ar[drr]_{\dot q_S}\ar[rr]^{j_S}&&
\hat X_S\ar[d]_{\hat q_S}&& \mathbb W_S\ar[ll]_{i}\ar[lld]^{\mathrm{pr}_S} &\\
     && S  &\\    }
\end{equation}
of $S$-schemes, which is a diagram of the weak elementary $\dot{Z}_S$-fibration $\dot{q}_S$.
For $U=\mathrm{Spec}\, \mathcal{O}_{X,\mathbf{x}}$, we take $\mathcal Z:= U\times_S \dot{Z}_S$ and $\mathcal X:= U\times_S \dot{X}_S$.
Hence, we obtain a smooth $U$-morphism $p_U\colon \mathcal{X}\to U$ of pure relative dimension one.
The canonical embedding $\dot{\mathrm{can}}: U\hookrightarrow \dot{X}_S$ induces a section $\Delta\colon U\rightarrow \mathcal X$. 
Hence, for the morphism $p_X\colon \mathcal X\rightarrow X$, we have $p_X\circ \Delta=\mathrm{can}$, where $\mathrm{can}: U\hookrightarrow X$ is the canonical embedding.
Clearly, the closed subscheme $\mathcal{Z}\subset \mathcal{X}$ is $U$-finite.
Particularly,  we get two commutative squares
\[
\xymatrix{
    \mathcal{X}- \mathcal{Z} \ar@{^{(}->}[r] \ar[d] & \mathcal{X} \ar[d]^{p_{X}} & \mathcal{Z} \ar@{_{(}->}[l] \ar[d] \\
X- Z \ar@{^{(}->}[r]           & X                  & Z \ar@{_{(}->}[l]
}
\]
It remains to construct the Nisnevich square in Theorem \ref{Moving_for_DVR}. 
First, we overcome the finite field obstruction for embedding $\mathcal Z$ into $\mathbb A^1_U$.
By Lemma \ref{overcome}, there is a smooth morphism $\tilde{\mathcal{X}}\to U$ of pure relative dimension one with a finite morphism $\phi\colon \tilde{\mathcal{X}}\to \mathcal{X}$ such that $\tilde{\mathcal{Z}}:= \phi^{-1}(\mathcal{Z}\cup \Delta(U))$ satisfies the following properties
\begin{itemize}
\item[(i)] $\tilde{\mathcal{Z}}$ is a  $U$-finite closed subscheme of $\tilde{X}$;
\item[(ii)] $\Delta$ lifts to a section $\Delta'\colon U\to \tilde{X}'$, namely $\Delta= \phi\circ \Delta'$;
\item[(iii)] $\phi|_{\tilde{\mathcal{Z}}}\colon \tilde{\mathcal{Z}}\twoheadrightarrow \mathcal{Z}$ is finite \'etale surjective;
\item[(iv)] there is no finite field obstruction for embedding $\tilde{\mathcal Z}$ into $\mathbb{A}_U^1$.
\end{itemize}
In particular, $\Delta'(U)\subset \tilde{\mathcal Z}$ and for $\tilde{p}_X\colon \tilde{X}\rightarrow X$, we have $\tilde{p}_X\circ \Delta'=p_X\circ\phi \circ \Delta'= p_X\circ \Delta=\mathrm{can}$.
Consequently, by Proposition \ref{map-to-affine-line}, there is a finite flat surjective $U$-morphism $\tau\colon \tilde{\mathcal{X}}\rightarrow \mathbb A^1_U$ such that $\tau|_{\tilde{\mathcal Z}}\colon \tilde{\mathcal Z}\hookrightarrow \mathbb A^1_U$ is a closed embedding and $\tau$ is \'etale at $\tilde{\mathcal Z}$.
Moreover, since there is no finite field obstruction, when we embedding $\tilde{\mathcal Z}$ we can make sure that $\tau(\tilde{\mathcal Z})\cap (\{1\}\times U)=\emptyset$.
Note that $\tau$ is a finite flat cover, so we let $\mathcal{X}^{\circ}$ be the complement of the branch locus in $\tilde{\mathcal X}$, then $\tau |_{\mathcal X^{\circ}}\colon \mathcal X^{\circ}\rightarrow \mathbb A^1_U$ is \'etale.
Since $\tau$ is \'etale at $\tilde{\mathcal Z}$, we know that $\tilde{\mathcal Z}\subset \mathcal X^{\circ}$.
By the property (ii) in Proposition \ref{map-to-affine-line}, if $\tau^{-1}(\tau(\tilde{\mathcal Z}))=\tilde{\mathcal Z}\sqcup \tilde{\mathcal Z}'$, then $\Delta(U)\cap \tilde{\mathcal Z}'=\emptyset$.
Now we put $\mathcal X^{\circ\circ}:= \mathcal X^{\circ}-\tilde{\mathcal Z}'$, then 
\begin{itemize}
\item[(a)] $\tau\colon \mathcal X^{\circ\circ}\rightarrow \mathbb A^1_U$ is \'etale;
\item[(b)] $\Delta'(U)\subset \tilde{\mathcal Z}\subset \mathcal X^{\circ\circ}$ and $\tau|_{\tilde{\mathcal Z}}\colon \tilde{\mathcal Z}\hookrightarrow \mathbb A^1_U$ is a closed embedding;
\item[(c)] $\tau^{-1}(\tau(\tilde{\mathcal Z}))=\tilde{\mathcal Z}$ in $\mathcal X^{\circ\circ}$.
\end{itemize}
In particular, we obtain an \textbf{\'etale} morphism $\tau\colon \mathcal{X}^{\circ\circ}\to \mathbb{A}^1_U$ such that $\tau|_{\tilde{\mathcal{Z}}}\colon \tilde{\mathcal{Z}}\hookrightarrow \mathbb{A}^1_U$ is a closed embedding, fitting into the commutative diagram with the two leftmost Cartesian squares
\[
    \xymatrix{
\mathbb{A}^1_U-\tau(\tilde{\mathcal{Z}}) \ar@{^{(}->}[d]      & \mathcal{X}^{\circ\circ}-\tilde{\mathcal{Z}} \ar@{^{(}->}[d] \ar[l] \ar@{^{(}->}[r]         & \tilde{\mathcal X}-\tilde{\mathcal Z} \ar@{^{(}->}[d] \ar[r] & \mathcal{X}-\mathcal{Z} \ar[r]  \ar@{^{(}->}[d] & \dot{X}_S-\dot{Z}_S \ar@{^{(}->}[d] \ar[r] & X-Z \ar@{^{(}->}[d] \\
\mathbb{A}^1_U                       & \mathcal{X}^{\circ\circ} \ar[l]^{\tau} \ar@{^{(}->}[r]                & \tilde{\mathcal X} \ar[r]^{\phi}       &   \mathcal X    \ar[r]      & \dot{X}_S \ar[r]                         & X                   \\
\tau(\tilde{\mathcal Z}) \ar@{^{(}->}[u] & \tilde{\mathcal{Z}} \ar@{^{(}->}[u] \ar[l]_{\sim} \ar[r]^{=} & \tilde{\mathcal{Z}} \ar@{^{(}->}[u] \ar[r]^{\phi|_{\tilde{\mathcal Z}}} & \mathcal Z \ar[r] \ar@{^{(}->}[u]  & \dot{Z}_S \ar@{^{(}->}[u] \ar[r]         & Z \ar@{^{(}->}[u] 
    }
\]
In particular, the upper left square is a Nisnevich square in $\mathbf{Sm}_U$, so after replacing $\mathcal X$ by $\mathcal{X}^{\circ\circ}$, $\mathcal Z$ by $\tilde{\mathcal Z}$, $\Delta$ by $\Delta'$, we finish the proof. \qedhere
\end{proof}

\end{document}